\documentclass[3p,times]{article}
\usepackage{makeidx}  
\usepackage{amsmath,amssymb,amsthm}
\usepackage{color}
\usepackage[all]{xy}
\usepackage{graphicx}
\usepackage{fancybox}
\usepackage{mathrsfs}
\usepackage{enumerate}
\usepackage{stmaryrd}
\usepackage{mathtools}
\usepackage{url}
\usepackage{lineno}
\usepackage[figuresright]{rotating}
\usepackage{fullpage}

\newcommand{\fsets}{{\bf Sets}_{\omega}}

\newcommand{\C}{\mathscr{C}}
\newcommand{\E}{\mathscr{E}}
\newcommand{\ff}{\mathscr{F}}
\newcommand{\G}{\mathscr{G}}

\newcommand{\cl}{\mathscr{B}_f}

\newcommand{\idp}{\mathfrak{p}}
\newcommand{\idq}{\mathfrak{q}}

\newcommand{\ido}{\mathfrak{o}}
\newcommand{\idf}{\mathfrak{f}}
\newcommand{\ida}{\mathfrak{a}}
\newcommand{\idb}{\mathfrak{b}}

\newcommand{\idP}{\mathfrak{P}}

\newcommand{\idA}{\mathfrak{A}}

\newcommand{\ratf}{\mathbb{Q}}

\newcommand{\fld}{\mathbb{F}}

\newcommand{\F}{\mathrm{F}}

\newcommand{\integer}{\mathbb{Z}}

\newcommand{\nat}{\mathbb{N}}
\newcommand{\idd}{\mathfrak{d}}

\newtheorem{thm}{Theorem}[section]
\newtheorem{lem}{Lemma}[section]
\newtheorem{prop}{Proposition}
\newtheorem{cor}{Corollary}

\theoremstyle{definition}
\newtheorem{defn}{Definition}
\newtheorem{ex}{Example}

\theoremstyle{remark}
\newtheorem{rem}{Remark}

\numberwithin{equation}{section}
\begin{document}
\title{Semi-galois Categories II:\\ An arithmetic analogue of Christol's theorem}
\author{Takeo Uramoto\\Graduate School of Information Sciences, Tohoku University}


\maketitle
\begin{abstract}
In connection with our previous work on \emph{semi-galois categories} \cite{Uramoto16,Uramoto17}, this paper proves an arithmetic analogue of \emph{Christol's theorem} concerning an automata-theoretic characterization of when a formal power series $\xi = \sum \xi_n t^n \in \fld_q[[t]]$ over finite field $\fld_q$ is algebraic over the polynomial ring $\fld_q[t]$. There are by now several variants of Christol's theorem, all of which are concerned with rings of positive characteristic. This paper provides an arithmetic (or $\fld_1$-) variant of Christol's theorem in the sense that it replaces the polynomial ring $\fld_q[t]$ with the ring $O_K$ of integers of a number field $K$ and the ring $\fld_q[[t]]$ of formal power series with the ring of Witt vectors. We also study some related problems. 
\end{abstract}
\section{Introduction}
\label{s1}
\noindent
The purpose of this paper is to prove an arithmetic analogue of the following theorem due to Christol \cite{Christol,CKFG}: Let $\fld_q$ be the finite field of $q$ elements and denote by $\fld_q[[t]]$ the ring of formal power series over $\fld_q$; also by $\fld_q[t] \subseteq \fld_q[[t]]$ the polynomial ring. Then \emph{Christol's theorem} states as follows (see \S \ref{s3s1} for undefined terminology here):

\begin{thm}[Christol's theorem]
 A formal power series $\xi = \sum \xi_n t^n \in \fld_q[[t]]$ is algebraic over $\fld_q[t]$ if and only if the coefficients $(\xi_n) \in \fld_q^{\nat_{\geq 0}}$ can be generated by some deterministic finite automaton (cf.\ \S \ref{s3s1}). 
\end{thm}

\noindent
This theorem lies at the intersection of number theory and automata theory; in the number theoretic context, automata theory typically has provided ideas to classify number-theoretic objects (such as formal power series \cite{Christol,CKFG,Kedlaya}, or expansions of real numbers by integer base \cite{Cobham,Hartmanis_Stearns,Adamczewski_Bugeaud1} and contined fraction \cite{Adamczewski_Bugeaud2}) from the viewpoint of computational complexity: Automata theorists ask how computationally complex it is to produce sequences of coefficients of formal power series or expansions of real numbers; and measure their complexity in terms of hierarchies of computational models that generate them. Christol's theorem is one of pioneering results of this direction, which characterized algebraicity of formal power series over $\fld_q[t]$ in terms of \emph{deterministic finite automata}--- the simplest computational models among others (such as Turing machines); this theorem provided a reasonable criterion to study transcendence of formal power series \cite{Allouche_Thakur}.

Several variants of Christol's theorem have been investigated in the literature, all of which were concerned with rings of positive characteristic. (See \cite{Allouche_Shallit} for more information on Christol's theorem and its variants.) The current paper is concerned with proving an \emph{arithmetic (or $\fld_1$-) analogue} of Christol's theorem in the sense that the polynomial ring $\fld_q[t]$ is replaced with the ring $O_K$ of integers of a number field $K$; in this variant, the ring $\fld_q[[t]]$ of formal power series is replaced with the ring of \emph{Witt vectors} \cite{Borger1} (cf.\ \S \ref{s2}). More formally, our variant of Christol's theorem claims as follows (see \S \ref{s3s2} for undefined terminology here): 

\begin{thm}[Arithmetic analogue of Christol's theorem]
 A Witt vector $\xi = (\xi_\ida) \in W_{O_K}(O_{\bar{K}})$ is integral over $O_K$ if and only if its coefficients $(\xi_\ida) \in O_{\bar{K}}^{I_K}$ can be generated by some deterministic finite automaton (cf.\ \S \ref{s3s2}). 
\end{thm}

As in the case of original Christol's theorem, the technical core of this proof is to estimate the size of the orbits of integral Witt vectors $\xi$ under the actions of (infinitely many) \emph{Frobenius liftings} $\psi_\idp$ that are canonically equipped to $W_{O_K}(O_{\bar{K}})$ as it forms a \emph{$\Lambda$-ring} (cf.\ \S \ref{s2}). For effective estimation, we will develop in \S \ref{s3} several basic bounds on coefficients of Witt vectors. In the same way as original one, our arithmetic analogue of Christol theorem relates the integrality of Witt vectors $\xi \in W_{O_K}(O_{\bar{K}})$ over $O_K$ and the automata-theoretic complexity of their coefficients; we say that a Witt vector $\xi=(\xi_\ida)$ is \emph{automatic} if its coefficients $\xi_\ida$ can be generated by some finite automaton. 

After this proof, we combine our Christol theorem with a strong classification result of certain $\Lambda$-rings due to Borger and de Smit \cite{Borger_Smit1,Borger_Smit2}, to conclude an explicit description of the integral closure of $O_K$ within $W_{O_K}(O_{\bar{K}})$ (\S \ref{s4}). To be specific, Borger and de Smit proved in \cite{Borger_Smit2} that the category $\C_K$ of those $\Lambda$-rings which are finite {\'e}tale over $K$ and have \emph{integral models} (cf.\ \S \ref{s2}) is dually equivalent to the category $\cl DR_K$ of \emph{finite $DR_K$-sets}, i.e.\ finite sets equipped with continuous actions of the profinite monoid $DR_K$ called the \emph{Deligne-Ribet monoid} \cite{Deligne_Ribet}; for this proof, they used class field theory in an essential way. The Deligne-Ribet monoid $DR_K$ is presented in a rather explicit way, that is, in terms of \emph{moduli} $\idf$ of $K$ (cf.\ \S \ref{s4}); and thanks to this presentation, together with our Christol theorem, we deduce the following characterization of integral Witt vectors $\xi \in W_{O_K}(O_{\bar{K}})$ described only in terms of $K$ (cf.\ the second claim):
\begin{thm}
 (I) If a Witt vector $\xi=(\xi_\ida) \in W_{O_K}(O_{\bar{K}})$ is integral over $O_K$, then the $\Lambda$-ring $K \otimes O_K \langle \xi \rangle$ is finite {\'e}tale over $K$ and has an integral model. In particular, (II) a Witt vector $\xi=(\xi_\ida) \in W_{O_K}(O_{\bar{K}})$ is integral over $O_K$ if and only if the coefficients $\xi_\ida$ are periodic with respect to some fixed modulus $\idf$ of $K$ in the sense of \cite{Borger_hand} (cf.\ \S \ref{s4}).
\end{thm}
\noindent
In the first claim (I), $O_K\langle \xi \rangle$ denotes the $O_K$-algebra generated over $O_K$ by the orbit of $\xi$ under the actions of (infinitely many) Frobenius liftings $\psi_\idp$. Our Christol theorem is necessary to prove that $O_K\langle \xi \rangle$ is in fact finite over $O_K$ when $\xi$ is integral over $O_K$. The second claim (II) is an immediate but non-trivial consequence of (I) and \cite{Borger_Smit2}, which describes an explicit classification, in terms of moduli $\idf$ of $K$, of those Witt vectors which are integral over $O_K$. 


Finally, in the last section (\S \ref{s5}), we review these results from more generic viewpoints of category theory and \emph{Eilenberg theory} \cite{Eilenberg} (cf.\ \S \ref{s5s2}) separately--- the latter is aimed to make smoother the connection of the current work with our previous one \cite{Uramoto16,Uramoto17}. On one hand, from the standpoint of number theory, (our analogue of) Christol's theorem can be seen as a result that classifies number-theoretic objects (i.e.\ Witt vectors) in terms of their computational complexity; and as we discuss in \S \ref{s5s1}, the equivalence $\C_K \simeq (\cl DR_K)^{op}$ of \cite{Borger_Smit2} can be seen as a categorical principle behind such Christol-type phenomena (i.e.\ phenomena that some number-theoretic objects can be generated by finite automata). In order to review this result from a more general standpoint, we make a few technical remarks on such a categorical equivalence. In this relation, we discuss a generic construction of such categorical equivalences, which is applicable e.g.\ to \emph{$\Lambda$-schemes} \cite{Borger} in particular. On the other hand, our arithmetic analogue of Christol theorem is also intended to develop a geometric intuition for a right extension of Eilenberg theory \cite{Eilenberg}, particularly via the concrete example $\C_K$ of semi-galois categories of \cite{Borger_Smit1,Borger_Smit2}\footnote{Some background on Eilenberg theory and its relationship to semi-galois categories will be briefly summarized in \S \ref{s5s2} for the neccessity of our discussion. For more comprehensive texts on Eilenberg theory and its historical aspect, the reader is referred to e.g.\ \cite{Pin} and \cite{Eilenberg}; also, for the axiomatization of Eilenberg theory based on semi-galois categories, see \S 1, \S 5 and \S 7 in \cite{Uramoto17}.}; we will discuss this matter carefully, taking some naive analogy (between differential and discrete goemetries) seriously, based on our results developed in this paper and the concept of \emph{arithmetic derivations} due to Buium \cite{Buium} (cf.\ \emph{Remark} \ref{derivation}, \S \ref{s2s2}). There we compare the relationship between linear and nonlinear differential equations with that between finite automata and more general computational models (e.g.\ Turing machines). In this relation we will discuss a common discrete-geometric mechanism that makes several computational models (e.g.\ cellular automata) Turing complete. See \S \ref{s5s2} for a more careful discussion.


We are grateful to Isamu Iwanari, who told us the papers \cite{Borger_Smit1,Borger}, and also to Go Yamashita, who continuously encouraged us. This work was supported by JSPS KAKENHI Grant number JP16K21115.

\section{Preliminaries}
\label{s2}
\noindent
For the sake of reader's convenience, we recall here some necessary concepts and constructions concerning \emph{Witt vectors} and \emph{$\Lambda$-rings} \cite{Borger1} over a Dedekind domain. In this paper we consider only a special case of general definitions of these concepts; for more complete theory, the reader is referred to the original work \cite{Borger1}. No novel result is presented here; our contribution starts from the next section (\S \ref{s3}). (Hence the reader who is already familiar with these concepts can go directly to \S \ref{s3}.)

\subsection{Witt vectors}
\label{s2s1}
\noindent
We begin with recalling the construction of the ring of (generalized) Witt vectors over Dedekind domain due to Borger \cite{Borger1}. Let $R$ be a Dedekind domain, whose residue fields at (non-zero) prime ideals are always assumed to be finite, and let $K$ be the field of fractions; typical examples of such Dedekind domains include rings $O_K$ of integers of number fields $K$. In this paper, mainly intending $K$ to be a number field and $R$ to be the ring $O_K$ of integers of $K$, we (abusively) denote for a Dedekind domain $R$ (and for its field $K$ of fractions) by $P_K$ the set of (non-zero) prime ideals of $R$, and by $I_K$ the monoid of (non-zero) ideals of $R$, so that $I_K$ is the commutative monoid freely generated by $P_K$. For $\idp \in P_K$, $k_\idp$ denotes the residue field $R / \idp$. Also, we denote by $\#k_\idp$ the cardinality of $k_\idp$. 

The rings of \emph{Witt vectors} are, in this paper, defined only for \emph{flat} $R$-algebras $A$. (In general, the definition needs one more step; cf.\ \cite{Borger1}). Let $A$ be a flat $R$-algebra and $A^{I_K}$ be the ($I_K$-times) product of $A$, whose $R$-algebra structure is defined component-wise. For each $\idp \in P_K$, denote by $\psi_\idp: A^{I_K} \rightarrow A^{I_K}$ the $R$-algebra endomorphism given by the shift $\psi_\idp ( (\xi_\ida)_\ida) := (\xi_{\idp\ida})_\ida$. Then, we define the sub-$R$-algebras $U_n(A) \subseteq A^{I_K}$ by induction on non-negative integers $n$ as follows:
\begin{eqnarray}
  U_0 (A) &:=& A^{I_K}; \\
  U_{n+1}(A) &:=& \bigl\{ \xi \in U_n (A) \mid \forall \idp \in P_K. \hspace{0.1cm} \psi_\idp \xi - \xi^{\#k_\idp} \in \idp U_n(A) \bigr\}. 
\end{eqnarray}
With these data:
\begin{defn}[the ring of Witt vectors]
The ring $W_R(A)$ of \emph{Witt vectors with coefficients in $A$} is defined as follows:
\begin{eqnarray}
  W_R(A) &:=& \bigcap_{n=0}^\infty U_n(A). 
\end{eqnarray}
The elements $\xi = (\xi_\ida)$ of $W_R(A)$ are called \emph{Witt vectors} (over $R$ with coefficients in $A$); and each $\xi_\ida$ of $\xi$ is called the \emph{coefficient} (or \emph{component}) of $\xi$ at $\ida \in I_K$. 
\end{defn}
\begin{rem}[some intuition on $W_R(A)$]
A natural intuition on this inductive construction of $W_R(A)$ may be better understood when it is phrased in terms of \emph{arithmetic analogue of derivations} by Buium \cite{Buium}, which will be briefly described after the definition of \emph{$\Lambda$-rings} (\S \ref{s2s2}). In some sense, the concept of arithmetic analogue of derivations allows us to regard $U_n(A)$ (resp.\ $W_R(A) = \bigcap_n U_n(A)$) as the ring of ``$C^n$-functions'' (resp.\ the ring of ``$C^\infty$-functions''). See \emph{Remark} \ref{derivation}, \S \ref{s2s2}.
\end{rem}
In what follows, if we need not specify the base ring $R$ for $W_R(A)$, we omit the subscript and write $W(A)$ instead of $W_R(A)$. Concerning rings of Witt vectors, the following proposition is fundamental and will be used throughout this paper:
\begin{prop}[Borger, Proposition 1.9 \cite{Borger1}]
 For each $\idp,\idq \in P_K$ and $\xi \in W(A)$, one has that (i) $\psi_\idp \xi \in W(A)$--- that is, $W(A)$ is closed under the shift $\psi_\idp: A^{I_K} \rightarrow A^{I_K}$; (ii) $\psi_\idp \psi_\idq = \psi_\idq \psi_\idp$; and:
\begin{eqnarray}
  \psi_\idp \xi &\equiv& \xi^{\#k_\idp} \hspace{0.1cm} \mod \idp W(A), 
\end{eqnarray}
that is, (iii) $\psi_\idp \xi - \xi^{\#k_\idp} \in \idp W(A)$. 
\end{prop}

\begin{ex}[Borger, \S 1.10, \cite{Borger1}]
\label{explicit presentation}
 The construction of $W_R(A)$ is not quite explicit, but when $R=\integer$ and $A=\integer$, the structure of $W_R(A)$ was described in \cite{Borger1} more explicitly as follows: 
\begin{eqnarray} 
  W_\integer(\integer) &=& \bigl\{ (\xi_n) \in \integer^{\nat} \mid \xi_{pn} \equiv \xi_n \mod p^{1+v_p(n)} \hspace{0.1cm} (\forall n \in \nat, \forall p: \textrm{prime}) \bigr\},
\end{eqnarray}
where $v_p(n)$ denotes the maximum index $e$ such that $p^e$ divides $n$. For the necessity of our purpose, we will extend later this expression of the ring $W_R(R)$ of Witt vectors from the case $R=\integer$ to the case of the ring $R=O_K$ of integers of an arbitrary number field $K$ (\S \ref{s3}). There, we will also present a partial description of $W_{O_K}(A)$ for the case where $A$ is the ring $O_L$ of integers of a finite extension $L / K$.
\end{ex}

\subsection{$\Lambda$-rings}
\label{s2s2}
\noindent
The above proposition says that the ring $W_R(A)$ of Witt vectors is an example of \emph{$\Lambda$-rings} (over $R$), which we now quickly recall; in fact, the construction $A \mapsto W(A)$ of rings of Witt vectors is a free construction of $\Lambda$-rings, which are defined as follows in the flat case (see \S 1.17, \cite{Borger1} for more general case):
\begin{defn}[$\Lambda$-ring]
 A \emph{$\Lambda$-ring} (or \emph{$\Lambda_R$-ring}) is a flat $R$-algebra $A$ equipped with $R$-algebra endomorphisms $\psi_\idp: A \rightarrow A$ given for each $\idp \in P_K$ that satisfies the following properties:
\begin{eqnarray}
  \psi_\idp \psi_\idq &=& \psi_\idq \psi_\idp; \\
  \psi_\idp a &\equiv& a^{\#k_\idp} \hspace{0.1cm} \mod \idp A, 
\end{eqnarray}
that is, $\psi_\idp a - a^{\#k_\idp} \in \idp A$. 
\end{defn}

\begin{ex}
 The ring $W_R(A)$ of Witt vectors is a $\Lambda_R$-ring with the $\Lambda$-operators $\psi_\idp$. 
\end{ex}

\begin{ex}
 Let $R=\integer$ be the ring of integers, whence the field of fractions $K$ is the rational number field $\ratf$ and $P_K$ is identified with the set of prime numbers. Then, for each natural number $N \geq 1$,  the ring $A_N:=\integer[z]/(z^N - 1)$ forms a $\Lambda_\integer$-ring, with $\Lambda$-ring operators $\psi_p: A_N \rightarrow A_N$ given by $\psi_p z := z^p$ for each prime number $p$. 
\end{ex}

\begin{rem}[notation $\psi_\ida$]
 By the commutativity $\psi_\idp \psi_\idq = \psi_\idq \psi_\idp$ of $\psi_\idp$'s on $\Lambda$-ring $A$, we can define $\psi_\ida: A \rightarrow A$ for each $\ida \in I_K$ by $\psi_\ida := \psi_{\idp_1} \circ \cdots \circ \psi_{\idp_n}$ if $\ida$ decomposes as $\ida = \idp_1 \cdots \idp_n$. We use this notation throughout this paper. 
\end{rem}

\begin{rem}[the maximality of $W_R(A)$]
\label{maximality of W}
For each $R$-algebra $A$, the $\Lambda$-ring $W_R(A)$ of Witt vectors is characterized as the largest sub-$\Lambda$-ring of $A^{I_K}$. In fact, assume that $B \subseteq A^{I_K}$ is a sub-$\Lambda$-ring. Then we have the inclusions $B \subseteq U_n (A)$ for every $n \geq 0$, which can be shown by induction on $n$: The case $n=0$ is trivial because $U_0(A) = A^{I_K}$; if the inclusion $B \subseteq U_n(A)$ is proved, then for every $\xi \in B$ we have $\psi_\idp \xi - \xi^{\#k_\idp} \in \idp B \subseteq \idp U_n(A)$ by the assumption that $B$ is a $\Lambda$-ring. This means that, by definition of $U_{n+1}(A)$, we have the inclusion $B \subseteq U_{n+1}(A)$; hence $B \subseteq \bigcap_n U_n(A) = W(A)$. This remark will be used to determine the structure of $W_{O_K}(O_K)$ for number fields $K$ (\S \ref{s3}). 
\end{rem}

\begin{rem}[integral model]
 As in \cite{Borger_Smit1,Borger_Smit2}, we will consider $\Lambda_R$-rings $X$ defined over the field $K$ of fractions of $R$ (rather than over $R$), e.g.\ the rings $\ratf[z]/(z^N - 1)$ with $\psi_p(z) = z^p$. When this is the case, i.e.\ when $X$ is a $K$-algebra, the second axiom of $\Lambda$-rings that $\psi_\idp x - x^{\#k_\idp} \in \idp X$ holds for arbitrary $R$-algebra endomorphisms $\psi_\idp: X \rightarrow X$. In other words, such $\Lambda_R$-rings are just $K$-algebras $X$ equipped with commuting family $\{\psi_\idp\}$ of $R$-algebra endomorphisms $\psi_\idp: X \rightarrow X$; so the second axiom is vain in this case. 

In particular, we say that a $\Lambda_R$-ring $X$ finite over $K$ \emph{has an integral model} if there exists a sub $\Lambda_R$-ring $A \subseteq X$ such that (i) $A$ is finite over $R$; and (ii) $X \simeq K \otimes_R A$, whence $A$ is called an \emph{integral model} of $X$. For instance $X:=\ratf[z]/(z^N -1)$ is an example of such rings, where $A:=\integer[z]/(z^N - 1)$ is an integral model. 
\end{rem}

\begin{rem}[Frobenius lifting and arithmetic derivation]
\label{derivation}
The map $\psi_\idp: A \rightarrow A$ satisfying $\psi_\idp a - a^{\#k_\idp} \in \idp A$ for all $a \in A$ in the definition of $\Lambda$-rings is conventionally called the \emph{Frobenius lifting} (of the Frobenius map $A/\idp A \ni a \mapsto a^{\#k_\idp} \in A/\idp A$), and naturally corresponds to \emph{arithmetic analogue of derivations} due to Buium \cite{Buium}. To see this, consider the case $K=\ratf$ (and $R=\integer$) here. 

Let $A$ be a $\Lambda_\integer$-ring. Then by definition of $\Lambda$-rings, for each prime number $p$ and each $a \in A$, there exists $a' \in A$ such that $\psi_p a = a^p + p a'$ (because $\psi_p a - a^p \in p A$). In particular, if $p$ is invertible in $A$, then one has $a' = (\psi_p a - a^p)/p$ in $A$; this operation $a \mapsto a'=(\psi_p a - a^p)/p$ was studied by Buium \cite{Buium} as an arithmetic analogue of the usual derivation $f(z) \mapsto f'(z)= \lim_{h \to 0} (f(z+h)-f(z))/h$ of functions. In what follows let us denote as $\delta_p a := (\psi_p a - a^p)/p$; then the arithmetic derivation $\delta_p$ satisfies the following properties similar to (but slightly different from) usual derivations (i.e.\ linearlity and Leibnitz rule):
\begin{eqnarray*}
 \delta_p(a + b) &=& \delta_p a + \delta_p b + C_p(a,b); \\
 \delta_p(ab) &=& (\delta_p a) b^p + a^p (\delta_p b) + p (\delta_p a) (\delta_p b);
\end{eqnarray*}
where $C_p(x,y) := (x^p + y^p - (x+y)^p)/p$. Also, one has $\delta_p(1) = 0$, which is analogue to the fact that the derivation of a constant function is constantly $0$. Conversely, given an operation $\delta_p:A \rightarrow A$ satisfying these properties, then the map $A \ni a \mapsto a^p + p\delta_p a \in A$ becomes a Frobenius lifting. So the Frobenius lifting $\psi_p: A \rightarrow A$ (which is a ring endomorphism on $A$) corresponds to arithmetic derivation $\delta_p: A \rightarrow A$ (which is not even additive) in this way (i.e.\ $\psi_p a = a^p + p \delta_p a$). 

In view of this, let us return to the inductive construction of the ring $W_R(A) = \bigcap_n U_n(A)$ of Witt vectors (\S \ref{s2s1}); again consider the case $R=\integer$. Recall that the first $U_0(A)$ is defined as $U_0(A) := A^{\nat}$, which makes it possible to consider the ring endomorphisms $\psi_p: U_0(A) \rightarrow U_0(A)$ by the shift $(\xi_n) \mapsto (\xi_{pn})$. Then, for each $n$, the $(n+1)$-st $U_{n+1}(A)$ is defined so that $\xi \in U_n(A)$ belongs in $U_{n+1}(A)$ if and only if $\psi_p \xi - \xi^p \in p U_n(A)$; that is, its derivation $\xi' = (\psi_p \xi - \xi^p)/p$ also belongs in $U_n(A)$. Therefore, if we regard $U_0(A)$ as the ring of ``$C^0$-functions'', then $U_n(A)$'s can be regarded inductively as those of ``$C^n$-functions'' in the sense of arithmetic derivations. Consequently, as $W_\integer(A)$ is the intersection of $U_n(A)$'s, the ring $W_\integer(A)$ of Witt vectors is the ring of ``$C^\infty$-functions'' in this sense.

For a general $R$, this intuitive description is not quite precise because its prime ideals $\idp$ are not always principal. But viewing the ring $W_R(A)$ of Witt vectors in this informal way will be helpful to gain some intuition on the following argument and constructions on Witt vectors and $\Lambda$-rings. In particular, the Frobenius liftings $\psi_\idp$ should be regarded as a sort of differential (or difference) operators.
\end{rem}

\section{Christol's theorem and its arithmetic analogue}
\label{s3}
\noindent
From now we start our contribution. After a review of the original Christol theorem and related concepts (\S \ref{s3s1}), we formulate and prove an arithmetic analogue of Christol's theorem (\S \ref{s3s2}). The next section (\S \ref{s4}) will then relate this variant of Christol theorem and the work of Borger and de Smit \cite{Borger_Smit1,Borger_Smit2}, which will be reviewed from the axiomatic standpoint of \emph{semi-galois categories} \cite{Uramoto16,Uramoto17} (\S \ref{s5}). 

\subsection{Original theorem}
\label{s3s1}
\noindent
Let $\fld_q$ be the finite field of $q$ elements. As mentioned in \S \ref{s1}, Christol's theorem states as follows:
\begin{thm}[Christol's theorem]
 A formal power series $\xi = \sum \xi_n t^n \in \fld_q[[t]]$ is algebraic over $\fld_q[t]$ if and only if the coefficients $(\xi_n) \in \fld_q^{\nat_{\geq 0}}$ can be ``generated by some deterministic finite automaton.''
\end{thm}
\noindent
In order to make this claim precise, we first describe the formal definitions of \emph{deterministic finite automata (with output)} (DFAs (resp.\ DFAOs)) and how they generate formal power series. 

Some conventional terminologies in automata theory are in order here: \emph{Alphabets} mean arbitrary sets, whose elements are called \emph{letters}; in this paper we do \emph{not} assume alphabets to be finite in general. A finite sequence $a_1 a_2 \cdots a_n$ of letters $a_i$ in a fixed alphabet $\Delta$ is called a \emph{finite word over $\Delta$}; in particular, the \emph{empty word} is the word of length $0$ and denoted $\varepsilon$. The set of finite words over an alphabet $\Delta$ forms a monoid with respect to the concatenation $(u,v) \mapsto uv$ of finite words; this monoid is denoted $\Delta^*$, where the empty word $\varepsilon$ is the unit. 

The concept of \emph{deterministic finite automata (with output)}, for short \emph{DFAs (resp.\ DFAOs)}, is defined as follows, which will be used to formalize Christol's theorem as well as our variant of the theorem: 

\begin{defn}[DFA and DFAO]
\label{dfa}
 Let $\Delta$ be a fixed alphabet. A \emph{deterministic finite automaton (DFA) over $\Delta$} is defined as a tupple $\mathfrak{A}=(S,\delta,s_0)$ of the following data:
\begin{enumerate}
 \item a finite set $S$ of \emph{states};
 \item a \emph{transition function} $\delta: S \times \Delta \rightarrow S$;
 \item an \emph{initial state} $s_0 \in S$. 
\end{enumerate}
A \emph{deterministic finite automaton with output (DFAO)} is a DFA $\mathfrak{A}=(S,\delta,s_0)$ equipped with an \emph{output function} $\tau:S \rightarrow O$ with values in a fixed set $O$ of \emph{output values}; we write DFAOs in such a way as $\mathfrak{A}_\tau=(S,\delta,s_0,\tau)$. The alphabet $\Delta$ over which the DFAO is defined is sometimes called the \emph{input alphabet}, while the set $O$ of output values is called the \emph{output alphabet} of the DFAO. 
\end{defn}

\begin{rem}
 The transition function $\delta: S \times \Delta \rightarrow S$ of a DFA $\mathfrak{A}=(S,\delta,s_0)$ is naturally extended to (and identified with) a map $\delta^*:S \times \Delta^* \rightarrow S$, which is given by induction on length of finite words as follows:
\begin{eqnarray*}
  \delta^* (s,\varepsilon) &:=& s; \\
  \delta^* (s,ua) &:=& \delta(\delta^*(s,u), a),
\end{eqnarray*}
where $s \in S$, $u \in \Delta^*$ and $a \in \Delta$. In what follows, we identify $\delta^*$ with $\delta$. Also, we often write as $\delta(s,u) = s\cdot u$ or simply $s u$ for $s \in S$ and $u \in \Delta^*$. 
\end{rem}

\begin{rem}[DFAO as graph]
\label{dfa as graph}
 It is conventional and helpful to picture DFAs (or DFAOs) as a kind of directed graphs. First, in general, let $\mathfrak{A}=(S,\delta,s_0)$ be a DFA over an alphabet $\Delta$. If $\delta(s,a) = s'$ for $s, s' \in S$ and $a \in \Delta$, we write this fact as $s \xrightarrow{a} s'$. In this way $\mathfrak{A}$ defines a finite directed graph, i.e.\ whose vertex set is $S$ and two vertices $s, s' \in S$ are connected by an edge $s \xrightarrow{a} s'$ with a label $a \in \Delta$ if and only if $\delta(s,a) = s'$. Therefore, in terms of the extended transition function $\delta^*$ above, we have $\delta^*(s,u)=s'$ for $s,s' \in S$ and $u=a_1\cdots a_n \in \Delta^*$ if and only if there exists a path $s \xrightarrow{a_1} s_1 \xrightarrow{a_2} \cdots \xrightarrow{a_n} s'$ in the graph. 

Similarly, a DFAO $\mathfrak{A}_\tau = (S,\delta,s_0,\tau)$ defines a finite directed graph; in this case, not only edges but also vertices as well are labelled by elements of the output alphabet $O$; namely, each vertex $s \in S$ of the graph corresponding to the DFA $\mathfrak{A}=(S,\delta,s_0)$ is labelled by the output value $\tau(s) \in O$. 
\end{rem}

Now we return to Christol's theorem. Firstly, in general, a DFAO defines a map of the form $\Delta^* \rightarrow O$ in the following way: Let $\mathfrak{A}_\tau = (S,\delta,s_0,\tau)$ be a DFAO with input alphabet $\Delta$ and output alphabet $O$. Given a finite word $u \in \Delta^*$, we have an output value $f_{\mathfrak{A}_\tau}(u) \in O$ defined by:
\begin{eqnarray}
 f_{\mathfrak{A}_\tau}(u)   &:=& \tau (\delta(s_0, u)).
\end{eqnarray}
Graphically, a DFAO outputs a value $b \in O$ from an input $u=a_1\cdots a_n$ if starting from the initial state $s_0$, transiting the states $s_0 \xrightarrow{a_1} s_1 \xrightarrow{a_2} \cdots \xrightarrow{a_n} s_n$ in the corresponding graph along the input $u$, then it terminates at the state $s_n$ which is labeled by $b = \tau(s_n)$. In this way, a DFAO gives a machinery to transform finite words $u$ over input alphabet $\Delta$ into an output value $b=f_{\mathfrak{A}_\tau}(u) \in O$. 

The generation of coefficients $(\xi_n) \in \fld_q^{\nat}$ of formal power series $\xi=\sum \xi_n t^n \in \fld_q[[t]]$ by finite automata, mentioned in the claim of Christol's theorem, is based on this machinery of DFAOs; to define it precisely, recall \emph{base-$p$ expansions} of natural numbers: Let $q$ be the cardinality of $\fld_q$, that is, a power $q=p^f$ of its characteristic $p$. As is well-known in elementary number theory, every natural number $n \in \nat$ can be uniquely expressed in the following form:
\begin{eqnarray}
  n &=& a_0 + a_1 p + a_2 p^2 + \cdots + a_r p^r,
\end{eqnarray}
where $a_i \in \{0,1, \cdots, p-1\} \simeq \fld_p$ and $a_r\neq 0$. Using this base-$p$ expression, we denote by $[n]_p$ the finite word $a_0 a_1 \cdots a_r \in \fld_p^*$ over the alphabet $\Delta=\fld_p$. 

Now we can define the concept of \emph{automatic formal power series} as follows: 

\begin{defn}[automatic series]
 We say that a formal power series $\xi=\sum \xi_n t^n \in \fld_q[[t]]$ is \emph{($p$-) automatic} if there exists a DFAO $\mathfrak{A}_\tau = (S,\delta,s_0,\tau)$ with input alphabet $\Delta=\fld_p$ and output alphabet $O=\fld_q$ that \emph{generates} the coefficients $\xi_n \in \fld_q$ in the following sense:
\begin{eqnarray}
  \xi_n &=& f_{\mathfrak{A}_\tau} ([n]_p),
\end{eqnarray}
for every $n \in \nat$. 
\end{defn}

Therefore, automatic series are those $\xi \in \fld_q[[t]]$ whose coefficients $(\xi_n)$ are controled by some fixed DFAO $\mathfrak{A}_\tau$, while coefficients of general formal power seires are arbitrary. Although the automaticity of formal power series seemingly has nothing to do with the algebraicity over the polynomial ring $\fld_q[t]$, Christol's theorem asserts that these concepts are precisely equivalent, which we restate as follows:
\begin{thm}[Christol's theorem, restated]
 A formal power series $\xi = \sum \xi_n t^n \in \fld_q[[t]]$ is algebraic over $\fld_q[t]$ if and only if it is automatic. 
\end{thm}

The important direction of this proof is to show the automaticity from the algebraicity of formal power series $\xi \in \fld_q[[t]]$; that is, one needs to construct a DFAO that generates a given $\xi \in \fld_q[[t]]$ when it is algebraic over $\fld_q[t]$. For this purpose, some typical proofs of Christol's theorem first consider the $\fld_p$-linear operators $\rho_i: \fld_q[[t]] \rightarrow \fld_q[[t]]$ for each $0 \leq i \leq p-1$ given by:
\begin{eqnarray}
  \rho_i \bigl(\sum_n \xi_n t^n \bigr) &:=& \sum_n \xi_{pn+i}^{1/p} t^n;
\end{eqnarray}
and then prove that the orbit of $\xi$ under these operators, i.e.\ $\bigl\{\rho_{i_1} \rho_{i_2} \cdots \rho_{i_r} (\xi) \in \fld_q[[t]] \mid 0 \leq i_1,\cdots,i_r \leq p-1 \bigr\}$ is in fact a finite set. See e.g.\ the recent proof due to Bridy \cite{Bridy} (following Speyer \cite{Speyer}), where the author also gave a sharp estimate of the size of DFAO that generates algebraic $\xi \in \fld_q[[t]]$ by relating the operators $\rho_i$ on $\fld_q[[t]]$ and \emph{Cartier operators} on differentials $\Omega_{K/\fld_q}$ on the algebraic curve defined by $\xi$ over $\fld_q$, and by using Riemann-Roch theorem for function fields over finite fields.
\subsection{Arithmetic analogue}
\label{s3s2}
\noindent
Our analogue of Christol's theorem is \emph{arithmetic} in the sense that the polynomial ring $\fld_q[t]$ is replaced with the ring $O_K$ of integers of a number field $K$. In this variant, the ring $\fld_q[[t]]$ of formal power series is replaced with the ring $W_{O_K}(O_{\bar{K}})$ of Witt vectors (cf.\ \S \ref{s2}) over the Dedekind domain $O_K$ with coefficients in algebraic integers $O_{\bar{K}}$. There is, however, a significant difference from the original Christol theorem, which can be illustrated most effectively when $K=\ratf$. 

In the original Christol theorem, to create a finite word $[n]_p$ as an input to a DFAO, natural numbers $n$ are decomposed by base-$p$ expansion, $n=a_0 + a_1 p + a_2 p^2 + \cdots + a_r p^r$; and then one inputs to DFAO the resulting finite word $[n]_p = a_0 a_1 a_2 \cdots a_r$ over $\fld_p=\{0,1,\cdots, p-1\}$, to output coefficients $(\xi_n) \in \fld_q^{\nat}$ of automatic formal power series $\xi = \sum \xi_n t^n \in \fld_q[[t]]$. To the contrary, our analogue is somewhat \emph{multiplicative} in the sense that we decompose natural numbers $n$ by \emph{prime factorization}, $n = p_1 p_2 \cdots p_r$; and input to our DFAOs this finite word $p_1 p_2 \cdots p_r$ over prime numbers, to output coefficients of \emph{automatic Witt vectors} $(\xi_n) \in W_\integer(O_{\bar{\ratf}})$ (see below). Hence, in the case $K=\ratf$, our DFAOs are defined over the input alphabet $P_\ratf$ and output alphabet $O_{\bar{\ratf}}$. More formally we define \emph{automatic Witt vectors} as follows in general.

Let $R$ be a Dedekind domain (with finite residue fields) and $K$ be the field of fractions of $R$. Recall that $P_K$ (abusively) denotes the set of (non-zero) prime ideals of $R$ and $I_K$ the set of (non-zero) ideals of $R$ respectively. Recall also that, for an $R$-algebra $A$, the ring $W_R(A)$ of Witt vectors over $R$ with coefficients in $A$ is defined as a sub-$R$-algebra of the product $A^{I_K}$, namely $W_R(A) \subseteq A^{I_K}$. Hence, in particular, Witt vectors can be represented as families $\xi=(\xi_\ida)$ of elements $\xi_\ida$ in $A$ given for each ideals $\ida \in I_K$. 

\begin{defn}[automatic Witt vector]
  We say that a Witt vector $\xi=(\xi_\ida) \in W_R(A)$ is \emph{automatic} if there exists a DFAO $\mathfrak{A}_\tau=(S,\delta,s_0,\tau)$ over input alphabet $\Delta=P_K$ and output alphabet $O=A$ such that, for each finite word $u=\idp_1 \idp_2 \cdots \idp_r$ over $P_K$, we have:
\begin{eqnarray}
  \xi_\ida &=& f_{\mathfrak{A}_\tau}(u),
\end{eqnarray}
for $\ida \in I_K$ that decomposes as $\ida = \idp_1 \idp_2 \cdots \idp_r$. 
\end{defn}

\begin{rem}
As shown in the following development, automatic Witt vectors in fact constitute a \emph{sub $\Lambda_R$-ring} of $W_R(A)$. In what follows, we generally denote by $W_R^a(A) \subseteq W_R(A)$ the sub $\Lambda$-ring of automatic Witt vectors. 
\end{rem}

The primary goal of this section is to prove \emph{the arithmetic analogue of Christol's theorem}, which, for technical simplicity, we first formulate in the following ``finitary'' form (cf.\ Corollary \ref{generic arithmetic analogue} below); combining with \cite{Borger_Smit2}, this theorem will be used to give another detailed description of integral Witt vectors in \S \ref{s4}. 

\begin{thm}[arithmetic analogue of Christol's theorem\footnote{Although this theorem is actually true even when $O_K$ is replaced with an arbitrary Dedekind domain (and perhaps Krull domains) with finite residue fields and $L/K$ with finite separable extensions as apparently seen from our proof, we focus on the case of number fields, which is necessary and sufficient for our purpose in \S \ref{s4}, where we relate our variant of Christol's theorem and the work of Borger and de Smit \cite{Borger_Smit1,Borger_Smit2}.}]
\label{arithmetic analogue}
 Let $L/K$ be a finite extension of number fields. A Witt vector $\xi \in W_{O_K}(O_L)$ is integral over $O_K$ if and only if it is automatic. 
\end{thm}

As in the case of original Christol's theorem, the technical core of this proof is to estimate the size of the orbit of (integral) Witt vector $\xi \in W_{O_K}(O_L)$ under the action of the operators $\psi_\idp$; actually, the fact that we have infinitely many operators $\psi_\idp$ makes the problem more delicate. For effective estimation, we develop several bounds on coefficients of Witt vectors. 

In the following, fix a number field $K$, which will be a base for other finite extensions $L/K$. For $\ida \in I_K$ and $\idp \in P_K$, we denote by $v_\idp(\ida)$ the maximum index $e$ such that $\idp^e$ divides $\ida$. Also, for a real number $x$, we denote by $\lfloor x \rfloor$ the maximum integer $n$ such that $n \leq x < n+1$. In what follows, we simply write as $W(A) := W_{O_K}(A)$ for $O_K$-algebras $A$ if we do not need to specify the base $K$; if the subscript is ommitted as $W(A)$, it should be understood as $W_{O_K}(A)$ for the fixed $K$. 

\begin{lem}
\label{bound on coefficients}
 For every $\xi \in W(O_L)$, we have the following congruence for each $\ida\in I_K$, $\idp \in P_K$, and $\idP \in P_L$ over $\idp$ with the inertia degree $f=f_{\idP\mid \idp}$:
\begin{eqnarray}
   \xi_{\idp^f \ida} &\equiv& \xi_\ida \hspace{0.1cm} \mod \idP^{1+\lfloor v_\idp(\ida)/ f \rfloor}. 
\end{eqnarray}
\end{lem}
\begin{proof}
 We prove this claim by induction on $n=\lfloor v_\idp(\ida) / f \rfloor$. For the base case, let $n=0$, namely $0 \leq v_\idp(\ida) < f$. Since $W(O_L)$ is a $\Lambda$-ring, we generally have $\psi_\idp \xi - \xi^{\#k_\idp} \in \idp W(O_L)$, from which we can deduce $\psi_{\idp^f} \xi - \xi^{\#k_\idp^f} \in \idp W(O_L)$. In particular:
\begin{equation}
  \xi_{\idp^f \ida} - \xi_\ida^{\#k_\idp^f} \in \idp O_L \subseteq \idP. 
\end{equation}
Moreover, since $\#k_\idp^f$ is the cardinality of the residue field $O_L / \idP$, we also have:
\begin{equation}
   \xi_{\ida}^{\#k_\idp^f} - \xi_\ida \in \idP, 
\end{equation}
which imply that $\xi_{\idp^f \ida} \equiv \xi_\ida \mod \idP$ as requested. Assume for induction that the claim is true up to $\lfloor v_\idp (\ida') / f \rfloor \leq n$; and let $\lfloor v_\idp(\ida)/f \rfloor = n+1$. When this is the case, we can write as $\ida = \idp^f \ida'$ for some $\ida' \in I_K$ such that $\lfloor v_\idp (\ida')/f \rfloor = n$. Since $\psi_{\idp^f}\xi - \xi^{\#k_\idp^f} \in \idp W(O_L)$ as seen above, we obtain the following two equalities by looking at the $\ida$-th and $\ida'$-th components of $\psi_{\idp^f}\xi - \xi^{\#k_\idp^f}$: (Note that $\ida=\idp^f \ida'$.)
\begin{eqnarray}
  \xi_{\idp^f \ida} - \xi_\ida^{\#k_\idp^f} &=& \sum_{i=1}^m r_i \zeta_\ida^{(i)}; \\
  \xi_{\ida} - \xi_{\ida'}^{\#k_\idp^f} &=& \sum_{i=1}^m r_i \zeta_{\ida'}^{(i)},
\end{eqnarray}
for some $r_i \in \idp$ and $\zeta^{(i)} \in W(O_L)$. Subtracting the second equation from the first one, we obtain:
\begin{eqnarray}
  \xi_{\idp^f \ida} - \xi_\ida &=& (\xi_\ida^{\#k_\idp^f} - \xi_{\ida'}^{\#k_\idp^f}) + \sum_i r_i (\zeta_{\idp^f \ida'}^{(i)} - \zeta_{\ida'}^{(i)}). 
\end{eqnarray}
Concerning $\zeta^{(i)}$'s in the second term of the right hand side, we have the following membership because of $\lfloor v_\idp(\ida')/f \rfloor =n$ and the induction hypothesis:
\begin{equation}
  \zeta_{\idp^f \ida'}^{(i)} - \zeta_{\ida'}^{(i)} \in \idP^{1+n}.
\end{equation}
Therefore, combining with $r_i \in \idp \subseteq \idP$:
\begin{equation}
  \sum_i r_i (\zeta_{\idp^f \ida'}^{(i)} - \zeta_{\ida'}^{(i)}) \in \idP^{1+1+n} = \idP^{1+\lfloor v_\idp(\ida) / f \rfloor}.
\end{equation}
On the other hand, concerning the first term:
\begin{eqnarray}
 \xi_{\idp^f\ida'}^{\#k_\idp^f} - \xi_{\ida'}^{\#k_\idp^f} &=& (\xi_{\idp^f \ida'} - \xi_{\ida'}) \cdot \sum_{l=0}^{\#k_\idp^f - 1} \xi_{\idp^f \ida'}^l \xi_{\ida'}^{\#k_\idp^f - 1 - l}. 
\end{eqnarray}
By induction hypothesis, $\xi_{\idp^f \ida'} - \xi_{\ida'} \in \idP^{1+n}$; and hence, putting it as $\alpha \in \idP^{1+n}$, one has $\xi_{\idp^f \ida'} = \xi_{\ida'} + \alpha$ and thus the following equality:
\begin{eqnarray}
 \sum_{l=0}^{\#k_\idp^f -1} \xi_{\idp^f \ida'}^l \xi_{\ida'}^{\#k_\idp^f - 1 - l} &=& \#k_\idp^f \cdot \xi_{\ida'}^{\#k_\idp^f -1} + \alpha \cdot (\textrm{terms of $\xi_{\ida'}$ and $\alpha$}).
\end{eqnarray}
But, since $\#k_\idp^f \in \idP$ and $\alpha \in \idP$, it follows that this term is also in $\idP$. Therefore, $\xi_{\idp^f\ida'}^{\#k_\idp^f} - \xi_{\ida'}^{\#k_\idp^f} \in \idP \cdot \idP^{1+n} = \idP^{1+\lfloor v_\idp(\ida)/f \rfloor}$, which concludes that:
\begin{equation}
  \xi_{\idp^f \ida} - \xi_\ida \in \idP^{1+\lfloor v_\idp(\ida) / f \rfloor},
\end{equation}
as requested. This completes the proof. 
\end{proof}

Before working in the general case of finite extensions $L/K$, we first study the structure of the $\Lambda$-ring $W_{O_K}(O_K)$, i.e.\ the base case $L=K$; indeed this base case plays a key role also in the general case. As the first step of this, the following proposition gives a presentation of $W_{O_K}(O_K)$, which is an extension of the presentation of $W_\integer(\integer)$ (Example \ref{explicit presentation}, \S \ref{s2s1}) from the case of $K=\ratf$ to arbitrary number fields $K$; in particular, the presentation of $W_{O_K}(O_K)$ (for arbitrary number fields $K$) will be used to compare our target $\Lambda$-ring $W_{O_K}(O_L)$ with $W_{O_L}(O_L)$. 

\begin{prop}
\label{presentation}
 The $\Lambda$-ring $W(O_K) = W_{O_K}(O_K)$ has the following presentation:
\begin{eqnarray}
   W_{O_K}(O_K) &=& \bigl\{ \xi \in O_K^{I_K} \mid \xi_{\idp \ida} \equiv \xi_\ida \mod \idp^{1+v_\idp(\ida)} \hspace{0.1cm} (\forall \idp. \forall \ida.) \bigr\}. 
\end{eqnarray}
\end{prop}
\begin{proof}
 Let us denote by $V(O_K)$ the right hand side of this equation. The above lemma, applied to the case $L=K$, proves the inclusion $W(O_K) \subseteq V(O_K)$; here we prove the inverse inclusion $V(O_K) \subseteq W(O_K)$. To this end, it suffices to show that $V(O_K)$ itself is a $\Lambda$-ring, namely $\psi_\idp\xi - \xi^{\#k_\idp} \in \idp V(O_K)$ for every $\idp \in P_K$ and $\xi \in V(O_K)$ (cf.\ Remark \ref{maximality of W}, \S \ref{s2}). To see this, fix $\xi \in V(O_K)$ and $\idp \in P_K$ throught this proof. 

We first see that if we put $\eta:= \psi_\idp \xi - \xi^{\#k_\idp}$, then we have $\eta_{\idp \ida} - \eta_\ida \in \idp^{2+v_\idp(\ida)}$ for every $\ida \in I_K$. In fact:
\begin{eqnarray*}
  \eta_{\idp \ida} - \eta_\ida &=& (\xi_{\idp^2 \ida} - \xi_{\idp \ida}^{\#k_\idp}) - (\xi_{\idp \ida} - \xi_\ida^{\#k_\idp}) \\
  &=& (\xi_{\idp^2 \ida} - \xi_{\idp \ida}) - (\xi_{\idp \ida}^{\#k_\idp} - \xi_\ida^{\#k_\idp}).
\end{eqnarray*}
The first term $\xi_{\idp^2\ida} - \xi_{\idp \ida}$ of the last equation belongs in $\idp^{2+v_\idp(\ida)}$ by definition of $V(O_K)$ and $\xi \in V(O_K)$. Concerning the second term:
\begin{eqnarray}
 \xi_{\idp\ida}^{\#k_\idp} - \xi_{\ida}^{\#k_\idp} &=& (\xi_{\idp \ida} - \xi_{\ida}) \cdot \sum_{l=0}^{\#k_\idp - 1} \xi_{\idp \ida}^l \xi_{\ida}^{\#k_\idp - 1 - l}. 
\end{eqnarray}
By the same argument as the above lemma, it follows that the second term $\xi_{\idp\ida}^{\#k_\idp} - \xi_\ida^{\#k_\idp}$ is also in $\idp^{2+v_\idp(\ida)}$. By these facts, we indeed have $\eta_{\idp\ida} - \eta_\ida \in \idp^{2+v_\idp(\ida)}$. 

Now we show $\psi_\idp \xi - \xi^{\#k_\idp} \in \idp V(O_K)$. To this end, notice that $\idp V(O_K) = \bigcap_\idq \idp (O_{K,\idq} \otimes V(O_K))$, where $\idq$ ranges over all prime ideals of $O_K$ and $O_{K,\idq}$ denotes the localization of $O_K$ at $\idq$; in other words, it suffices to show that $\psi_\idp \xi - \xi^{\#k_\idp} \in \idp O_{K,\idq} \otimes V(O_K)$ for every $\idq \in P_K$. In the case of $\idq \neq \idp$, there is nothing to prove because $\idp O_{K,\idq} \otimes V(O_K) = O_{K,\idq} \otimes V(O_K)$. We see that $\psi_\idp \xi - \xi^{\#k_\idp} \in \idp O_{K,\idp} \otimes V(O_K)$. Let $\pi \in \idp$ be a uniformizer, whence $\idp O_{K,\idp} \otimes V(O_K) = \pi O_{K,\idp} \otimes V(O_K)$. Note that $\psi_\idp \xi - \xi^{\#k_\idp} \in \idp O_K^{I_K} \subseteq \pi O_{K,\idp}^{I_K}$ because $\xi_{\idp\ida} - \xi_\ida \in \idp$ (by $\xi \in V(O_K)$) and $\xi_\ida - \xi_\ida^{\#k_\idp} \in \idp$ for every $\ida \in I_K$. Hence, there exist a unit $u \in O_{K,\idp}^\times$ (more specifically, we take $u \in O_{K,\idp}^\times$ to be of the form $u=1/b$ with $b \in O_K\setminus \idp$) and $\zeta \in O_K^{I_K}$ such that:
\begin{eqnarray}
  \psi_\idp \xi - \xi^{\#k_\idp} &=& \pi u \cdot \zeta.
\end{eqnarray}
Now looking at the $\ida$-th component for each $\ida \in I_K$:
\begin{eqnarray}
  \xi_{\idp\ida} - \xi_\ida^{\#k_\idp} &=& \pi u \cdot \zeta_\ida.
\end{eqnarray}
By subtracting this equation for the $\ida$-th component from that for the $\idp\ida$-th component, we obtain:
\begin{eqnarray}
  \pi u \cdot (\zeta_{\idp\ida} - \zeta_\ida) &=& \eta_{\idp \ida} - \eta_\ida \in \idp^{2+v_\idp(\ida)}. 
\end{eqnarray}
Since we took $u$ from $O_{K,\idp}^\times = O_{K,\idp} \backslash \pi O_{K,\idp}$, this implies that $\zeta_{\idp \ida} - \zeta_\ida = (\pi u)^{-1}(\eta_{\idp\ida} - \eta_\ida) \in \idp^{1+v_\idp(\ida)}$. Also, applying the same argument for the $\idq\ida$-th component with $\idq \neq \idp$, we see that $\zeta_{\idq \ida} - \zeta_{\ida} = (\pi u)^{-1}(\eta_{\idq\ida} - \eta_\ida) \in \idq^{1+v_\idq(\ida)}$ as well, where the last membership to $\idq^{1+v_\idq(\ida)}$ follows from the fact that $\eta_{\idq\ida} - \eta_\ida \in \idq^{1+v_\idq(\ida)}$ by $\eta=\psi_\idp \xi - \xi^{\#k_\idp} \in V(O_K)$ and $(\pi u)^{-1} = b/\pi \in O_{K,\idq}$ (by $\idq\neq\idp$). In other words $\zeta \in V(O_K)$ by definition, hence $\psi_\idp \xi - \xi^{\#k_\idp} = \pi u \cdot \zeta \in \pi O_{K,\idp} \otimes V(O_K)$ as requested. This completes the proof. 
\end{proof}

\begin{rem}
 If a finite extension $L/K$ is non-trivial, the inclusion $W_{O_K}(O_L) \subseteq \{\xi \in O_L^{I_K} \mid \xi_{\idp^f \ida} \equiv \xi_\ida \mod \idP^{1+\lfloor v_\idp(\ida)/f \rfloor} \}$ (Lemma \ref{bound on coefficients}) is proper. In fact, if this inclusion is the equality, $O_L$ embeds in $W_{O_K}(O_L)$ by the diagonal map $a \mapsto (a)_\ida$. But when this is the case, it would follow (from Proposition \ref{integrality}, \S \ref{s4}) that the field $L$ is a finite {\'e}tale $\Lambda$-ring over $K$ with integral model on which $\psi_\idp$'s are all identities; this can happen only when $L=K$. 
\end{rem}

The following lemma is the key result for our proof of Theorem \ref{arithmetic analogue} in the base case $L=K$. 

\begin{lem}[orbit finiteness, the base case]
\label{orbit finiteness}
 If $\xi \in W(O_K)$ is a Witt vector whose set $\{\xi_\ida \in O_K \mid \ida \in I_K\}$ of coefficient values is a finite set, then the orbit $I_K \xi:= \{\psi_\ida \xi \in W(O_K) \mid \ida \in I_K\}$ under the action of $\psi_\idp$'s is also a finite set. 
\end{lem}
\begin{proof}
Fix $\xi \in W(O_K)$ that satisfies the assumption. Put $C_\xi :=\{\xi_\ida \mid \ida \in I_K\}$ and $T_\xi := \{ \xi_\ida - \xi_\idb \mid \ida,\idb \in I_K, \hspace{0.1cm} \xi_\ida \neq \xi_\idb\}$. Then we define $\idd_\xi=\idd \in I_K$ so that $v_\idp(\idd) = \max[\{0\}\cup \{ e \in \nat \mid \exists \zeta \in T_\xi.\hspace{0.1cm} \zeta \in \idp^e \}]$. Since $T_\xi$ is finite, this indeed defines a divisor $\idd \in I_K$. Concerning this $\idd$, notice that if $v_\idp(\idd) < v_\idp(\ida)$ for $\idp\in P_K$, then we have $\ida=\idp \ida'$ for some $\ida'$ (by $v_\idp(\ida) > v_\idp(\ido) \geq 0$) and the equality $\xi_{\ida} = \xi_{\ida'}$. In fact, by $\xi_{\idp\ida'} \equiv \xi_{\ida'} \mod \idp^{1+v_\idp(\ida')}$ (Proposition \ref{presentation}), we have $\xi_\ida - \xi_{\ida'} \in \idp^{1+v_\idp(\ida')} = \idp^{v_\idp(\ida)}$. But by definition of $\idd$ and $v_\idp(\idd) < v_\idp(\ida)$, $\xi_\ida - \xi_{\ida'}$ never belongs in $T_\xi$ (i.e.\ is necessarily zero), hence $\xi_\ida = \xi_{\ida'}$. Also, for every $\eta \in I_K \xi$, its components are in $C_\xi = \{\xi_\ida \mid \ida \in I_K\}$. These facts imply that $\eta \in I_K \xi$ is determined by the values of the components $\eta_\ida$ at those $\ida \in I_K$ such that $v_\idp(\ida) \leq v_\idp(\idd)$ for all $\idp \in P_K$, namely $\ida \mid \idd$.  Therefore, the size of $I_K \xi$ is not greater than that of $C_\xi^{\sigma(\idd)}$, where $\sigma(\idd) = \prod_\idp (1+v_\idp(\idd))$ denotes the number of those $\ida \in I_K$ dividing $\idd$, which is finite. This completes the proof. 
\end{proof}

With these preparations, we can now prove almost immediately at least the base case of our arithmetic analogue of Christol's theorem. We will use this case later, and state it as a lemma for more general case. For practical use, we restate this case of the theorem in the following refined form:

\begin{lem}[arithmetic analogue of Christol's theorem, the base case]
\label{base case of arithmetic analogue}
The following four conditions on $\xi \in W_{O_K}(O_K)$ are equivalent:
\begin{enumerate}
 \item $\xi$ is integral over $O_K$;
 \item the coefficients $C_\xi=\{ \xi_\ida \in O_K \mid \ida \in I_K\}$ form a finite set;
 \item the orbit $I_K\xi=\{\psi_\ida \xi \in W(O_K) \mid \ida \in I_K\}$ is a finite set;
 \item $\xi$ is automatic. 
\end{enumerate}
\end{lem}
\begin{proof}
The implications $(1) \Leftrightarrow (2)$ and $(4) \Rightarrow (2)$ are easy; the implication $(2) \Rightarrow (3)$ is proved in Lemma \ref{orbit finiteness}. To see $(3)\Rightarrow (4)$, assume that $I_K \xi$ is finite, from which we construct a DFAO $\mathfrak{A}_\tau = (S,\delta,s_0,\tau)$ that generates $\xi$. The state set $S$ is $I_K \xi$, and the transition function $\delta: I_K \xi \times P_K \rightarrow I_K \xi$ is given by $\delta(\eta,\idp) := \psi_\idp \eta$. The initial state $s_0 \in I_K \xi$ is $\xi$ itself, and the output function $\tau: I_K \xi \rightarrow O_K$ is given by $\tau(\eta) := \eta_1$, where $1$ denotes the unit ideal $(1)=O_K \in I_K$. Since we generally have $(\psi_\ida \eta)_\idb = \eta_{\ida \idb}$, it follows that $\tau(\delta(\xi,\ida)) = \xi_\ida$. This means that the DFAO constructed here indeed generates $\xi$. This completes the proof. 
\end{proof}

Based on this case, we prove the general case. Let $L/K$ be a finite extension, and in what follows, let $N_{L\mid K} = N: I_L \rightarrow I_K$ be the monoid homomorphism determined by the assignment $N\idP:= \idp^f$ for $\idP \in P_L$ and $\idp \in P_K$ such that $\idP \mid \idp$ with the inertia degree $f = f_{\idP \mid \idp}$. With this map $N$, for each $\xi \in O_L^{I_K}$, we can define $N^* \xi \in O_L^{I_L}$ as follows:
\begin{eqnarray}
  N^* \xi &:=& ( \xi_{N \idA} )_{\idA \in I_L}.
\end{eqnarray}
Concerning this assignment $N^*: O_L^{I_K} \rightarrow O_L^{I_L}$, we have the following: 

\begin{lem}
 If $\xi \in W_{O_K}(O_L)$, then $N^* \xi \in W_{O_L}(O_L)$; in particular, if $\xi \in W_{O_K}^a(O_L)$, then $N^* \xi \in W_{O_L}^a(O_L)$. 
\end{lem}
\begin{proof}
 Put $\zeta := N^* \xi$ for short. In order to show $\zeta \in W_{O_L}(O_L)$ based on Proposition \ref{presentation}, it suffices to prove that $\zeta_{\idP \idA} \equiv \zeta_\idA \mod \idP^{1+v_\idP(\idA)}$ for every $\idP \in P_L$ and $\idA \in I_L$. In terms of $\xi$, this means that, by $N(\idP \idA) = \idp^f N\idA$:
\begin{eqnarray}
  \xi_{\idp^f N\idA} &\equiv& \xi_{N\idA} \mod \idP^{1+v_\idP(\idA)}.
\end{eqnarray}
To see this, note that by Lemma \ref{bound on coefficients} and $\xi \in W_{O_K}(O_L)$, we have the following congruence for $\ida:=N\idA$:
\begin{eqnarray}
  \xi_{\idp^f \ida} &\equiv& \xi_\ida \mod \idP^{1+\lfloor v_\idp(\ida)/f \rfloor}.
\end{eqnarray}
But notice also that, $v_\idp(\ida) = v_\idp(N\idA) \geq f \cdot v_\idP(\idA)$; that is, $\lfloor v_\idp(\ida)/f \rfloor \geq v_\idP(\idA)$, which completes the proof of the first claim of this lemma. The second claim is immediate because $C_{N^*\xi} \subseteq C_\xi$ is finite for $\xi \in W^a_{O_K}(O_L)$ and by Lemma \ref{base case of arithmetic analogue} applied to $W^a_{O_L}(O_L)$. 
\end{proof}

\begin{lem}[orbit finiteness: generalization of Lemma \ref{orbit finiteness}]
\label{orbit finiteness 2}
 If $\xi \in W_{O_K}(O_L)$ is a Witt vector whose set $\{\xi_\ida \in O_L \mid \ida \in I_K\}$ of coefficient values is a finite set, then the orbit $I_K \xi:= \{\psi_\ida \xi \in W_{O_K}(O_L) \mid \ida \in I_K\}$ under the action of $\psi_\idp$'s is also a finite set. 
\end{lem}
\begin{proof}
We may assume that $L/K$ is galois without loss of generality because we have a $\Lambda_{O_K}$-ring injection $W_{O_K}(O_L) \hookrightarrow W_{O_K}(O_{L'})$ for an arbitrary extension $L \subseteq L'$. Let $\xi \in W_{O_K}(O_L)$ satisfy the condition. In order to estimate the size of $I_K \xi$, we first study the action of $\psi_\idp$'s on $\xi$ for unramified prime ideals $\idp \in P_K$: Let $\idp \in P_K$ be unramified in $L$ and $\ida \in I_K$ be arbitrary. Also let $\idP \in P_L$ be a prime ideal of $O_L$ over $\idp$, for which we denote by the Artin symbol $(\frac{L\mid K}{\idP}) \in G(L/K)$ its Frobenius automorphism. Then by definition of Frobenius automorphism:
\begin{eqnarray}
  \xi_\ida^{(\frac{L\mid K}{\idP})} &\equiv& \xi_\ida^{\#k_\idp} \mod \idP.
\end{eqnarray}
On the other hand, since $\xi \in W_{O_K}(O_L)$, we also have:
\begin{eqnarray}
  \xi_{\idp\ida} &\equiv& \xi_\ida^{\#k_\idp} \mod \idP.
\end{eqnarray}
Therefore, we get $\xi_{\idp \ida} - \xi_\ida^{(\frac{L\mid K}{\idP})} \in \idP$ for every unramified $\idp \in P_K$ and an arbitrary $\ida \in I_K$. Here, notice that the following set is a finite set, because $C_\xi = \{\xi_\ida \mid \ida \in I_K\}$ and $G(L/K)$ are finite sets:
\begin{eqnarray}
  H_\xi &:=& \bigl\{\xi_{\idp \ida} - \xi_\ida^{(\frac{L\mid K}{\idP})} \in O_L \mid \ida \in I_K, \textrm{$\idp$ is unramified in $L$} , \idP \mid \idp, \textrm{and $\xi_{\idp\ida} \neq \xi_\ida^{(\frac{L\mid K}{\idP})}$} \bigr\}. 
\end{eqnarray}
This finiteness implies that $H_\xi$ can intersect with only finitely many $\idP \in P_L$; therefore, for all but finitely many $\idp$ unramified in $L$, one must have the following \emph{equality} for all $\ida \in I_K$:
\begin{eqnarray}
\label{equality for bound}
  \xi_{\idp \ida} &=& \xi_\ida^{(\frac{L\mid K}{\idP})}.
\end{eqnarray}
Let us denote by $\idp'_1, \cdots, \idp'_r$ such exceptional unramified prime ideals. Also let $\idp_1,\cdots,\idp_s \in P_K$ be the prime ideals of $K$ \emph{ramified} in $L$, and put $E:=\{\idp_1,\cdots,\idp_s, \idp'_1, \cdots,\idp'_r\} = \{\idq_1,\cdots,\idq_N\}$. This argument means that, by the equality (\ref{equality for bound}), the values of all $\xi_\ida$'s are determined by the values at those $\ida\in I_K$ in which $\idp \notin E$ is not contained, i.e.\ $\ida$ is divisible only by $\idq_j \in E$. In addition to this fact, moreover, notice that (i) $I_K^d \subseteq N_{L\mid K} I_L \subseteq I_K$ with the degree $d=[L:K]$, (ii) $I_L\cdot N^*\xi$ is a finite set by Lemma \ref{base case of arithmetic analogue} applied to $L$, and (iii) $C_\eta \subseteq C_\xi$ for every $\eta \in I_K \xi$. Combining these facts, it then follows that $\eta \in I_K \xi$ is determined by the coefficients at those $\ida \in N_{L\mid K} I_L \subseteq I_K$ (the size of which are bounded by some $\ido \in I_L$ as seen in the argument of Lemma \ref{orbit finiteness} applied to $I_L \cdot N^*\xi$) and actions of $\psi_\idq$'s for finitely many $\idq \in E$, each of which is of finite order (at most $d$) modulo $N_{L\mid K} I_L$. This implies that $I_K \xi$ is indeed a finite set as requested. This completes the proof. 
\end{proof}

This lemma concludes the arithmetic analogue of Christol's theorem in the general case (Theorem \ref{arithmetic analogue}). As in the case of Lemma \ref{base case of arithmetic analogue}, we restate it in the following refined form. Since the proof is the same as Lemma \ref{base case of arithmetic analogue}, replacing Lemma \ref{orbit finiteness} with Lemma \ref{orbit finiteness 2}, we ommit the proof. 

\begin{thm}[arithmetic analogue of Christol's theorem]
\label{refined arithmetic analogue}
For every finite extension $L/K$, the following four conditions on $\xi \in W_{O_K}(O_L)$ are equivalent:
\begin{enumerate}
 \item $\xi$ is integral over $O_K$;
 \item the coefficients $C_\xi=\{ \xi_\ida \in O_L \mid \ida \in I_K\}$ form a finite set;
 \item the orbit $I_K\xi=\{\psi_\ida \xi \in W_{O_K}(O_L) \mid \ida \in I_K\}$ is a finite set;
 \item $\xi$ is automatic. 
\end{enumerate}
\end{thm}

Finally, furthermore, let $\bar{K}$ be the separable closure of $K$. Note that, if $\xi \in W_{O_K}(O_{\bar{K}})$ is automatic, its coefficients $\xi_\ida$ are all in $O_L$ for some finite extension $L/K$. Therefore, one finds that $\xi$ is automatic as an element of $W_{O_K}(O_L)$ as well. So we restate Theorem \ref{arithmetic analogue} also as in the following form:

\begin{cor}
\label{generic arithmetic analogue}
 A Witt vector $\xi \in W_{O_K}(O_{\bar{K}})$ is integral over $O_K$ if and only if it is automatic; and we have the following equality among the $\Lambda_{O_K}$-rings of Witt vectors:
\begin{eqnarray}
  W^a_{O_K}(O_{\bar{K}}) &=& \bigcup_{L/K} W_{O_K}^a(O_L). 
\end{eqnarray}
\end{cor}
\begin{proof}
 This can be proved by a similar argument to Proposition \ref{presentation}: that is, it suffices to show $O_L^{I_K} \cap W(O_{\bar{K}})$ is a $\Lambda$-ring, which implies $O_L^{I_K} \cap W(O_{\bar{K}}) \subseteq W_{O_K}(O_L)$; but this follows by a localization argument as in the case of Proposition \ref{presentation}. 
\end{proof}

\section{$\Lambda$-rings generated by automatic Witt vectors}
\label{s4}
\noindent
This section relates our result on automatic Witt vectors $\xi \in W_{O_K}(O_{\bar{K}})$ (\S \ref{s3}) and the strong classification result of $\Lambda_{O_K}$-rings by Borger and de Smit \cite{Borger_Smit1,Borger_Smit2}. Similarly to that Bridy \cite{Bridy} considered algebraic curves for formal power series $\xi \in \fld_q[[t]]$ algebraic over $\fld_q[t]$ in order to estimate the size of DFAOs generating them, we consider $\Lambda$-rings generated by integral Witt vectors $\xi \in W_{O_K}(O_{\bar{K}})$; here, it is aimed to obtain from \cite{Borger_Smit2} a strong description of coefficients of integral Witt vectors (cf.\ Corollary \ref{description of integral witt vectors}). 

To be precise, Borger and de Smit \cite{Borger_Smit1,Borger_Smit2} classified those $\Lambda_{O_K}$-rings that are finite {\'e}tale over $K$ and have integral models (cf.\ \S \ref{s2}), applying class field theory. Specifically, in terms of category theory, they proved that the category of such $\Lambda$-rings is dually equivalent to the category $\cl DR_K$ of finite $DR_K$-sets, where $DR_K$ denotes the \emph{Deligne-Ribet monoid}, a profinite monoid studied by Deligne and Ribet \cite{Deligne_Ribet} (cf.\ \S \ref{s5}). In elementary words of Witt vectors, their classification implies that (integral models of) every such a $\Lambda$-ring embeds into (finite product of) the $\Lambda$-ring consisting of those Witt vectors $\xi \in W_{O_K}(O_{\bar{K}})$ whose components $\xi_\ida$ ($\ida \in I_K$) are \emph{periodic} with respect to a fixed modulus $\idf \in I_K$ of $K$ in the following sense (cf.\ \cite{Borger_hand}):
\begin{eqnarray}
  \xi_\ida = \xi_\idb &\textrm{if}& \ida \sim_\idf \idb,
\end{eqnarray}
where we denote $\ida \sim_\idf \idb$ if there exists a totally positive number $t \in K$ such that $t \in 1+\idf \idb^{-1}$ and $\ida \idb^{-1} = (t)$, which is a monoid congruence on $I_K$ of finite index \cite{Deligne_Ribet}; this monoid congruence $\sim_\idf$ is used to define the Deligne-Ribet monoid $DR_K := \lim_\idf I_K/\sim_\idf$. 

It is clear that if the components $\xi_\ida$ of a Witt vector $\xi \in W_{O_K}(O_{\bar{K}})$ are periodic with respect to some $\idf \in I_K$ in the above sense, it is automatic (and integral) because $\sim_\idf$ is a monoid congruence of $I_K$ of finite index. The purpose of this section is to show the inverse: That is, if a Witt vector $\xi \in W_{O_K}(O_{\bar{K}})$ is integral over $O_K$ (automatic), then its coefficients $\xi_\ida$ must be periodic with respect to some modulus $\idf$ of $K$; in other words, the subring of $W_{O_K}(O_{\bar{K}})$ consisting of those Witt vectors whose components are periodic with respect to some $\idf$ coicides with the integral closure of $O_K$ within $W_{O_K}(O_{\bar{K}})$, which is precisely $W^a_{O_K}(O_{\bar{K}})$. 

To obtain this refined characterization of integral Witt vectors, it suffices to see the following general fact: (We ommit the subscripts $O_K$ from $W_{O_K}$.)

\begin{prop}
\label{integrality}
Let $\xi \in W(O_{\bar{K}})$ be integral over $O_K$ and denote by $O_K \langle \xi \rangle$ the $O_K$-algebra generated by the set $I_K \xi$ over $O_K$. Then the $\Lambda$-ring $K \otimes O_K\langle \xi \rangle$ is finite {\'e}tale over $K$ and has an integral $\Lambda$-model, the maximum one being given by $W^a(O_{\bar{K}}) \cap K \otimes O_K\langle \xi \rangle$. 
\end{prop}
\begin{proof}
  Throughout this proof, let us write $X=K \otimes O_K \langle \xi \rangle$ for short, which has natural actions of $\psi_\idp$'s to form a $\Lambda$-ring. Since the orbit $I_K \xi$ is finite by assumption (Theorem \ref{refined arithmetic analogue}) and its elements $\eta \in I_K \xi$ are all integral over $O_K$ ($I_K \eta \subseteq I_K \xi$), the $O_K$-algebra $O_K\langle \xi \rangle = O_K[I_K\xi]$ is finite over $O_K$. Moreover, since $O_K \langle \xi \rangle$ is reduced (and finite over $O_K$), it follows that $K \otimes O_K \langle \xi \rangle=X$ is finite {\'e}tale over $K$. Thus it suffices to show that $A:=W^a(O_{\bar{K}}) \cap K \otimes O_K\langle \xi \rangle$ is the maximum integral model of $X$. 

Firstly, since $O_K \langle \xi \rangle \subseteq A$, we have $X = K \otimes A$. Also, since $A$ is integral over $O_K$, $A$ is included in the integral closure of $O_K$ in $X$, which is finite over $O_K$; hence, so is $A$. Thus the rest task, other than the maximality of $A$, is to show that for every $\zeta \in W^a(O_{\bar{K}})$, we have $\psi_\idp \zeta - \zeta^{\#k_\idp} \in \idp W^a(O_{\bar{K}})$, which we prove by localization\footnote{Note that the constructions of this proof are closed in $K \otimes O_K\langle \xi \rangle$; so we focus on the part of $W^a(O_{\bar{K}})$ for simplicity of argument.}. Fix $\zeta \in W^a(O_{\bar{K}})$ and $\idp \in P_K$. First notice again that $ \idp W^a(O_{\bar{K}}) = \bigcap_{\idq \in P_K} \idp (O_{K,\idq} \otimes W^a(O_{\bar{K}}))$. So it suffices to see that $\psi_\idp \zeta - \zeta^{\#k_\idp} \in \idp (O_{K,\idq} \otimes W^a(O_{\bar{K}}))$ for every $\idq \in P_K$. In the case of $\idq \neq \idp$, there is nothing to prove because $\idp (O_{K,\idq} \otimes W^a(O_{\bar{K}})) = O_{K,\idq} \otimes W^a(O_{\bar{K}})$. We prove the case $\idq = \idp$, whence $\idp (O_{K,\idp} \otimes W^a(O_{\bar{K}})) = \pi O_{K,\idp} \otimes W^a(O_{\bar{K}})$ for a uniformizer $\pi \in \idp$. Since $\psi_\idp\zeta - \zeta^{\#k_\idp} \in \pi O_{K,\idp} \otimes W(O_{\bar{K}})$, there exist $\eta \in W(O_{\bar{K}})$ and $u \in O_{K,\idp}$ such that:
\begin{eqnarray*}
  \psi_\idp \zeta - \zeta^{\#k_\idp} &=& \pi u \cdot \eta.
\end{eqnarray*}
Therefore, we obtain $\{\eta_\ida \mid \ida \in I_K \}= \bigl\{(\pi u)^{-1} (\zeta_{\idp \ida} - \zeta_\ida^{\#k_\idp}) \mid \ida \in I_K \bigr\}$. But since $\zeta \in W^a(O_{\bar{K}})$, the latter set (hence the former one) is a finite set. This means that $\eta$ is automatic, thus $\psi_\idp \zeta - \zeta^{\#k_\idp} \in \pi O_{K,\idp} \otimes W^a(O_{\bar{K}})$ as requested. Finally, we see that $A$ is maximum among integral $\Lambda$-models of $X$. Let $B \subseteq X=K\otimes O_K\langle \xi \rangle$ be an integral $\Lambda$-model of $X$, and let $\chi \in B$. Since $\chi$ is integral over $O_K$, its components $\chi_\ida$ are all in $O_{\bar{K}}$ (rather than $\bar{K}$) and thus $\chi \in W(O_{\bar{K}})$. Again, since $\chi \in W(O_{\bar{K}})$ is integral over $O_K$, it is automatic, i.e.\ $\chi \in A$. Therefore, $B \subseteq A$, which completes the proof. 
\end{proof}

As an immediate but non-trivial consequence from \cite{Borger_Smit2}, we obtain the following (above-mentioned) description of integral Witt vectors:

\begin{cor}
\label{description of integral witt vectors}
A Witt vector $\xi \in W_{O_K}(O_{\bar{K}})$ is integral over $O_K$ if and only if its coefficients $\xi_\ida$ are periodic with respect to some modulus $\idf$ of $K$ (and also, if and only if it is automatic).
\end{cor}
\begin{proof}
 The if part is trivial; we see the only-if part. Let $\xi$ be integral. By the above proposition, the $\Lambda$-ring $X=K \otimes O_K \langle \xi \rangle$ is an object of the category $\C_K$ of $\Lambda_{O_K}$-rings that are finite {\'e}tale over $K$ and have integral models, which is equivalent to $(\cl DR_K)^{op}$ \cite{Borger_Smit2}. Thus the actions of $\psi_\idp$'s on $\xi \in X$ factor through the finite quotient $DR_K \twoheadrightarrow I_K/ \sim_\idf$ for some $\idf$. With respect to this $\idf$, the coefficients $\xi_\ida$ of $\xi$ must be periodic. 
\end{proof}

\section{Concluding remarks}
\label{s5}
\noindent
The major goal of this paper was to prove an arithmetic analogue of Christol's theorem (Theorem \ref{s3}, \S \ref{s3}); combining this variant of Christol's theorem with the result of Borger and de Smit \cite{Borger_Smit1,Borger_Smit2}, we then deduced a description of integral Witt vectors (Corollary \ref{description of integral witt vectors}, \S \ref{s4}). To review these results from a more general standpoint in this section (\S \ref{s5s1} -- \S \ref{s5s2}), we now conclude this paper with a few technical remarks on the above-mentioned equivalence $\C_K \simeq (\cl DR_K)^{op}$ of \cite{Borger_Smit2} and the condition of when one can prove such equivalence; we also provide a few generic example (\S \ref{s5s1}). Moreover, since this paper is a continuation of our previous work \cite{Uramoto16,Uramoto17}, the second subsection (\S \ref{s5s2}) is intended to make the connection between the current work and our previous work \cite{Uramoto16,Uramoto17} smoother; in parcular, since our reformulation of Eilenberg theory in \cite{Uramoto16,Uramoto17} was partially intended to seek a right direction to extend the theory from finite automata to more general computational classes such as those of Turing machines, we clarify the location of the current work and our previous work \cite{Uramoto16,Uramoto17} in a wider context of computational hierarchy. In this relation we also discuss a geometric framework to extend Eilenberg theory. The discussion in \S \ref{s5s2} focuses on the conceptual aspect of our framework; for this purpose and for effectiveness, our argument there is given in informal words. In particular, we assume that the reader is more or less familiar with the concept of Turing machines (cf.\ \cite{Sipser,Hopcroft_Ullman}).

\subsection{Christol's theorem and the equivalence $\C_K \simeq (\cl DR_K)^{op}$}
\label{s5s1}
\noindent
Borger and de Smit \cite{Borger_Smit2} proved the equivalence $\C_K \simeq (\cl DR_K)^{op}$ based on class field theory; and by this result, they established an interesting connection between class field theory and the theory of $\Lambda$-rings. Also, as mentioned above, this equivalence particularly implies that (integral models of) every $\Lambda$-ring $X \in \C_K$ embeds into (finite product of) those consisting of (periodic) Witt vectors. Here, to clarify the relationship between this equivalence and the result developed in this paper we remark that (I) the category $\cl DR_K$ canonically embeds into \emph{the category of DFAs over the alphabet $P_K$}, cf.\ \S 2.2, \cite{Uramoto17} (since $DR_K$ is topologically generated by $P_K$); and (II) the DFA corresponding to a given $X \in \C_K \simeq (\cl DR_K)^{op}$ under this embedding is exactly one which generates (with suitable output functions) those Witt vectors corresponding to elements of (integral models of) $X$. In other words, the equivalence $\C_K \simeq (\cl DR_K)^{op}$ explains how integral elements of each $X \in \C_K$ correspond to automatic Witt vectors. In relation to this, our concern in this paper was to characterize such (automatic) Witt vectors in \emph{automata-free} words, i.e.\ they are precisely integral Witt vectors, to reveal an arithmetic analogue of Christol's theorem (cf.\ \S \ref{s5s2}).

We further remark that a categorical equivalence itself $\C_K \simeq (\cl M)^{op}$ for \emph{some} profinite monoid $M$ can be proved without referring to class field theory; it is for the description of $M$ (as the Deligne-Ribet monoid $DR_K$) where class field theory is really needed. In fact it is not difficult to prove directly that (the opposite of) the category $\C_K$ satisfies the axiom of \emph{semi-galois categories} \cite{Uramoto16,Uramoto17}; and as we proved in \cite{Uramoto17}, every semi-galois category is equivalent to the category of the form $\cl M$ for some profinite monoid $M$ (called the \emph{fundamental monoid} of the semi-galois category). Although this abstract proof of the equivalence $\C_K \simeq (\cl M)^{op}$ does not tell about the structure of $M$\footnote{But it also should be mentioned here that it is possible, without using class field theory, to know the structure of $M$ to some degree by comparing the construction of the fundamental monoid $M$ (cf.\ \S 3, \cite{Uramoto17}) and the definition of (objects of) $\C_K$. In fact, for instance, we can directly prove that every abelian (i.e.\ galois) extension $L/K$ becomes a (maximal) component of (galois) object of $\C_K$ using easy ramification theory and Theorem 1 of \cite{Borger_Smit2} (and using the abstract constructions on galois objects, \S 3 \cite{Uramoto17}); from this it follows that the maximal abelian galois group $G_K^{ab}$ is a quotient of the unit group $M^\times$ (which is actually an isomorphism).}, an apparent merit is that it is readily extended to more general situations: That is, let $R$ be a Dedekind domain with finite residue fields and $K$ be the field of fractions. In the completely same way as the case of number fields, we can define the category, denoted $\C_K$ again, of $\Lambda_R$-rings that are finite {\'e}tale over $K$ and have integral models. Concerning this category, we also have the following classification result:
\begin{prop}
 There exists a profinite monoid $M_K$ such that $\C_K$ is opposite equivalent to $\cl M_K$. 
\end{prop}
\noindent
The proof is almost straightforward (hence not discussed here); for this proof one only uses the fact due to Borger \cite{Borger1} that the category of $\Lambda_R$-rings has finite limits and finite colimits so that the underlying rings coincide with those taken in the category of $R$-algebras. From this equivalence $\C_K \simeq (\cl M_K)^{op}$, it follows that $\C_K$ embeds into the category of DFAs over some (possibly infinite) alphabet $\Delta$ that generates $M_K$ topologically. (In other words $M_K$ is a quotient of the free profinite monoid $\widehat{\Delta^*}$.) In general, however, we do not know whether $M_K$ admits an explicit description similar to $DR_K$ (which is a quotient of $\widehat{P_K^*}$; and the fact that we can take $\Delta=P_K$ in the case of number fields $K$ is based on Chebotarev's density theorem, cf.\ \S 3 \cite{Borger_Smit2}). 

We also make a remark on a more generic (topos-theoretic) construction of categorical equivalence $\C \simeq \cl M$, apart from the context of Witt vectors and $\Lambda$-rings, but motivated by Borger's framework of geometry over $\fld_1$ based on \emph{$\Lambda$-algebraic geometry} (cf.\ \cite{Borger}). Let $\gamma: \ff \twoheadrightarrow \E$ be a \emph{surjective} geometric morphism (i.e.\ $\gamma^*$ reflects isomorphisms); an example of such geometric morphisms is the geometric morphism $\gamma: Sp_\integer \rightarrow Sp_{\fld_1}$ from the topos $Sp_\integer$ of spaces over $\integer$ to that $Sp_{\fld_1}$ of spaces over $\fld_1$ constructed by Borger\footnote{Borger's geometric morphism $\gamma: Sp_\integer \rightarrow Sp_{\fld_1}$ is not only surjective but also \emph{essential} (i.e.\ the inverse image functor $\gamma^*$ has a \emph{left} adjoint $\gamma_!$ in addition to the right adjoint $\gamma_*$); but at least for the construction of semi-galois categories, geometric morphisms $\gamma:\ff \twoheadrightarrow \E$ do not need to be essential.} (see \cite{Borger}). Let $S \in \E$ be a (connected) object of $\E$ (thought of as a ``space over $\fld_1$''), for which we define the category $\C_S$ as follows: The objects of $\C_S$ are those arrows $h: T \rightarrow S$ in $\E$ such that $\gamma^*h: \gamma^* T \rightarrow \gamma^* S$ is locally finite, and locally constant in $\ff$; and the arrows from $h: T \rightarrow S$ to $h': T' \rightarrow S$ are those $f:T \rightarrow T'$ in $\E$ such that $h'f=h$. With these data (and a suitable base point), we have the following:
\begin{prop}
 There exists a profinite monoid $M_S$ such that $\C_S$ is (covariantly) equivalent to $\cl M_S$. 
\end{prop}
\begin{proof}[Proof (sketch)]
 Let $\G_{\gamma^* S}$ be the galois category of locally finite and locally constant objects over $\gamma^* S \in \ff$, and $\F: \G_{\gamma^* S} \rightarrow \fsets$ be an arbitrarily chosen fiber functor (here assume its existence). By definition of $\C_S$ there exists an \emph{exact} functor $\C_S \rightarrow \G_{\gamma^* S}$ induced naturally from $\gamma^*: \E \rightarrow \ff$. Also, by the surjectivity of $\gamma$, this functor $\gamma^* : \C_S \rightarrow \G_{\gamma^* S}$ as well as $\F: \G_{\gamma^* S} \rightarrow \fsets$ reflects isomorphisms. Therefore, it is readily seen that the pair $\langle \C_S, \F \circ \gamma^* \rangle$ satisfies the axiom of semi-galois category (\S 2, \cite{Uramoto17}). Hence, in particular, $\C_S$ is equivalent to $\cl M_S$, where $M_S$ is given by the fundamental monoid $\pi_1(\C_S,\F \circ \gamma^*)$ of this semi-galois category.
\end{proof}

As observed from these examples, categorical equivalences of the form $\C \simeq \cl M$ appear naturally and can be proved by formal category-theoretic arguments, although the structure of $M$ here is mysterious only from this formal construction\footnote{Although we shall not discuss here the problem on what this fundamental monoid $M_S$ is for $S \in \E$ (which is beyond the subject of this paper itself), it may be worthwhile to notice from the proof that there exists a canonical homomorphism from the fundamental group $\pi_1(\G_{\gamma^* S},\F)$ of the galois category $\G_{\gamma^* S}$ to the fundamental monoid $M_S = \pi_1(\C_S,\F \circ \gamma^*)$. Therefore, the fundamental monoid $\pi_1(\C_S,\F \circ \gamma^*)$ may play some role also in understanding the fundamental group $\pi_1(\G_{\gamma^* S},\F)$ from a new angle; this should be compared to the relationship between the Deligne-Ribet monoid $DR_K$ (that is a limit of monoids $I_K/\sim_\idf$ of ideal classes) and the maximal abelian galois group $G_K^{ab}$.}. From the viewpoint of automata theory, such phenomena will provide a nice class of problems, say, the problem to classify number-theoretic (or geometric) objects (such as $\Lambda$-schemes) by comparing them with computational hierarchies of finite automata (and formal languages) based on the achievement in the classical Eilenberg theory (cf.\ \cite{Pin,Eilenberg}); indeed, historically speaking, this problem is the very central subject of Christol's theorem. In this paper, we studied this problem in the case of rings of Witt vectors, with clarifying and highlightening its relationship to our axiomatization of Eilenberg theory \cite{Uramoto16,Uramoto17} (cf.\ \S \ref{s5s2}).

\subsection{Canonicity of Eilenberg theory for DFAs and its geometric extension}
\label{s5s2}
\noindent
Finite automata constitute the simplest class of computational models. Informally speaking, among general computational models such as Turing machines, finite automata are characterized as those machines with finite controls whose next step at each computational stage is determined only by the current reading input and state, while general Turing machines can refer to their past behaviors via their read-and-write tapes. In this way Turing machines can realize recursive procedures, while the behaviors of finite automata are strictly bounded\footnote{These informal explanations of behaviors of Turing machines and finite automata can be of course formally stated in terms of purely combinatorial structures. But in our exposition, let us use these informal words because it seems most effective for the expository purpose here.}. There are plenty of variations of computational models (cf.\ e.g.\ \cite{Sipser,Hopcroft_Ullman}), including not only finite automata and Turing machines but those models which restrict the ability of Turing tapes or have other devices alternative to such tapes; each of them defines their respective computational hierarchy. Morally, the central subject of computability and complexity theories is then to understand the structure of hierarchies of these computational models by classifying or separanting them, which includes important open problems such as the $P\neq NP$ problem. 

On one hand, naively speaking, this diversity of computational models makes this field interesting and challenging in that for each specific class of computational models there are their respective technical issues on the dynamics of the models; therefore, modifying the architecture of the computational models, one gets a new class of problems to work with. Having said that, on the other hand, this diversity and arbitrarity of extensions (and modifications) of computational models contradicts to a goal of \emph{systematic} understanding of our computability concept (see also \emph{``the next 700 syndrome''}, \cite{Abramsky}); it seems that this situation prevents this field from being applicable deeply (yet systematically) to other fields such as number theory, despite that this field contains critical ideas to classify mathematical objects.

More formally, on one hand, it is indeed widely recognized that we have already achieved a consensus on an appropriate rigorous definition of the computability concept, thanks to the earlier works due to Turing, Church, Kleene and G{\"o}del in the 1930s. But on the other hand, we do not yet have--- if exists though--- a \emph{canonical} computational model, or at least a canonical understanding of computability concept, from which we are still far. In fact, for the time being, our consensus on the definition of computability concept is just based on (I) mutual equivalences between various computational models (say, Turing machines, $\lambda$-caluculi, recursive functions, generative grammars etc.) and (II) empirical facts that these formal models indeed can implement many of our real algorithms in their respective senses. Strictly speaking, however, it seems that what this equivalence of models actually proves is just the robustness of our computability concept, not the canonicity of these individual models; and more notably, equivalences between distinct computational models are usually proved in a quite ad hoc way and based on arbitrary (sometimes artificial) translations of codings of these models. 

In view of this, it is still meaningful to ask why these distinct computational models could be equivalent, or in other words, what their common principle is. Indeed this question is relevant to the current work and our previous work on an \emph{axiomatization} of Eilenberg theory \cite{Uramoto16,Uramoto17}: As discussed in \cite{Uramoto17}, Eilenberg theory \cite{Eilenberg} concerns a classification of certain hierarchies of finite automata (i.e.\ those called \emph{varieties} of finite automata in particular, cf.\ \S 5, \cite{Uramoto17}\footnote{To be more precise, Eilenberg theory classifies those hierarchies called varieties of \emph{finite actions}, rather than finite automata, which are essentially the same structure as finite automata $(S,\delta,s_0)$, except for that they do not have fixed initial states $s_0$. In some sense, Eilenberg theory dares to ignore the choice of initial states of finite automata because it makes their classification more transparent in that they admit a \emph{variety theorem}, cf.\ \S 5 \cite{Uramoto17}.}); and in a sense, our results in \cite{Uramoto17} state a \emph{canonicity} of these hierarchies of finite automata in that they can be axiomatized purely in terms of formal category-theoretic exactness conditions. In fact, as exemplified in this paper with Witt vectors, it is possible that a category $\C$ (say $\C_K$) whose objects seemingly have nothing to do with finite automata is canonically equivalent to that of finite automata; and (in the case of $\C_K$) this equivalence indicates the fact that every element $\xi \in X$ of each object $X \in \C$ of this category $\C$ can be generated by finite automata, namely, is automatic (cf.\ \S \ref{s5s1}). As mentioned in \S 1, this phenomenon lies precisely in the line of the natural intersection between number theory and automata theory, where one considers which machine can generate sequential (or approximational) representations of number-theoretic objects (say, fixed-base expansions $a_0.a_1a_2a_3\cdots$ of real numbers); put in this line, what we have shown in this paper with \cite{Uramoto16,Uramoto17} can be informally summarized as the fact that ``such generations of number-theoretic objects by machines can be canonically axiomatized in the case of finite automata.'' This gives an ideal form of automata theory, both from the viewpoint of canonicity and applicability. 

\paragraph{Geometric extension of Eilenberg theory}
From the viewpoint of generality, nevertheless, our current framework is restrictive for the time being in that its scope is focused only to the simplest machine classes, i.e.\ finite automata. To cope with this situation, yet keeping some sort of canonicity, we conclude this paper with a discussion on yet another framework, which is not categorical but \emph{geometric}, and seems extendable and sustainable naturally even for the case beyond finite automata\footnote{Nevertheless, this geometric consideration here is also intended to seek a right categorical direction to extend semi-galois categories. In relation to this and our consideration below, we are concerned with Malgrange's differential galois theory \cite{Malgrange} for non-linear differential equations, where galois groups are replaced with \emph{groupoids}. See below.}. We start with our Christol theorem, which suggests a framework where one represents classical computational models by \emph{arithmetic differential equations} in the sense of Buium \cite{Buium_geometry,Buium_differential}.

Putting aside the technical details, what we actually want to express by classical computational models such as Turing machines is usually very simple. For instance, the basic idea behind the formal definition of Turing machines (e.g.\ \cite{Sipser,Hopcroft_Ullman}) is just an intuitive idea of ``a machine that reads and writes symbols on a tape in its computational procedure.'' (See the original paper \cite{Turing} as well.) To formalize such an intuitive idea, it is conventional in computability theory to use discrete geometric structures (such as DFAs, cf.\ Definition \ref{dfa}, \S \ref{s3s1}). But the more we need to model complicated computational procedures, the more it becomes difficult and ad hoc to formalize them only by discrete-geometric structures, ending up with a diversity of computational models with little common principle. This situation is in a sharp contrast to that in physics: While the dynamical systems that physisits are concerned with are also complicated in the same level (or more) as computational (discrete) dynamical systems that we are concerned with, physisists have developed successful standard theories such as general relativity, quantum mechanics or more. Unarguably, the success of physics heavily relies on the existence of suitable geometric frameworks such as differential equations and differential geometry; these geometric frameworks are expressible enough to describe physical dynamical systems in a natural way (that is, one can deduce differential equations modelling physical systems by an intuitive manner, e.g.\ using the idea of infinitesimal). Unfortunately, unlike phisical systems, computational processes that we want to model (say, \emph{recursive} ones in particular) are discrete in nature; and because there is no naive concept of differentiations on discrete objects, we cannot directly simulate the method of physics for our computational dynamical systems in a naive way.

It is in this relation that we pay a special attention to our Christol theorem, where discrete structures (i.e.\ DFAs) appear not directly, but as a consequence of \emph{rigidity} of geometric structures (i.e.\ rigidity of integral Witt vectors; recall our proof of the orbit finiteness, cf.\ Lemma \ref{orbit finiteness}, \S \ref{s3s2}) on which one can consider a sort of differential operations (i.e.\ Frobenius liftings $\psi_\idp$; cf.\ \emph{Remark} \ref{derivation}, \S \ref{s2s2}); this philosophy is also related to what Borger emphasized in \cite{Borger}\footnote{Let us quote his argument in p.3, \cite{Borger}: ``\emph{the combinatorial nature is not built into the foundation of the theory--- it is a consequence of hard arithmetic results in presence of finiteness conditions.}'' (Here ``the theory'' is $\Lambda$-algebraic geometry \cite{Borger}.) In our context, the ``combinatorial nature'' corresponds to automata-theoretic objects; and what we are doing here is to relate combinatorial and geometric structures via rigidity (or descent). So we are concerned with both objects (actually the former).}. In this indirect way, discrete structures can appear with rich background structures; and as we  carefully discuss below, it seems reasonable to expect that this philosophy will apply, more or less, to further complicated computational dynamical systems beyond finite automata, actually even to \emph{Turing complete} ones.

Conventionally, the dynamics of higher computational models such as Turing machines is defined using some external register memories (such as Turing tapes); and it is the interactions with such external devices that make their behavior more complicated (or \emph{non-linear}) than finite automata: On one hand, the behavior of a finite automaton $\idA=(S,\delta,s_0)$ over an alphabet $\Delta$ is defined by the transition function $\delta: S \times \Delta \rightarrow S$; namely, if the automaton $\idA$ is at a state $s \in S$ and reads an input $a \in \Delta$, then the next state $s' \in S$ is given by $s'=\delta(s,a)$ (cf.\ \emph{Remark} \ref{dfa as graph}, \S \ref{s3s1}). Therefore, the behavior of a finite automaton is determined only by its current state $s \in S$ and the input $a \in \Delta$ that the automaton is reading now. On the other hand, the behavior of Turing machines (cf.\ \cite{Sipser,Hopcroft_Ullman}) is more complicated due to their interactions with their tapes: To be more specific, because (some version of) a Turing machine not only reads input from the tape but also rewrites it, and the pointer on the tape can move both to the left and right, the Turing machine may read a symbol which he wrote at his past computational stage. In other words, Turing machines can be influenced by their own past behaviors, which is a key machinery to realize recursive procedures. There are plenty of seemingly distinct versions even for Turing machines (using e.g.\ multiple tapes or registers of $2$-dimentional form); but they are equivalent, and therefore, we should pay a special attention not to individual implementations but to what we are actually doing with these models. 

We note that, in this informal exposition, we intentionally emphasized the term ``interactions'' in order to highlight a mechanism that one can observe more or less commonly in several distinct (Turing complete) computational models, such as \emph{cellular automata} in particular (at least if we widely construe the meaning of ``interactions''). In fact, as we saw in the case of Turing machines, (some sort of) interactions with their own past behaviors play a key role for computational models to realize recursive procedures, so as to be Turing complete; this viewpoint is more explicit in the definition of cellular automata, where each cell interacts with its neighboring cells; and the effect of this interaction transmits from cells to cells; by this machinery, cellular automata can implement the functions of Boolean circuits, so that they can be Turing complete (cf.\ \cite{cell}). Albeit at least in a naive conceptual level, we see that the basic idea behind these discrete dynamical systems can be compared to that behind physical systems, i.e.\ the interactions between physical entities (either of microscopic or macroscopic scales); and what we want to say with this lengthy conceptual argument is that we shall take this analogy seriously.

In fact, not just at a conceptual level, our Christol theorem can be regarded as a result of this direction, where ``differentiation'' is formalized by the concept of arithmetic derivation due to Buium (\emph{Remark} \ref{derivation}, \S \ref{s2s2}) in particular. From this viewpoint, our result concerns the case of integral solutions $\xi$ in the ring $W_{O_K}(O_{\bar{K}})$ of Witt vectors; and they can be also seen as solutions of systems of simple \emph{linear} differential (or difference) equations due to their automaticity. The major reason why we especially highlighted this viewpoint here is that (i) it can be seamlessly extended from linear to non-linear case, and (ii) this extension from linear to non-linear case seems to capture some essential aspect of the extension from the dynamics of finite automata to that of more general computational models in the sense that we discussed with Turing machines above\footnote{The situation in the study of non-linear differential equations may also explain why we need to have plenty of computational models. But in this relation, the idea of Malgrange's differential galois theory for non-linear differential equations \cite{Malgrange} may be important for our purpose here and its categorical axiomatization, where galois groups are replaced with groupoids. In view of that we replaced groups with monoids, this may naively suggest us to replace groupoids with categories.}. In order to make this naive viewpoint solid, we need to investigate whether Turing complete computational discrete dynamical systems can be represented by (some class of) non-linear arithmetic differential equations in some way, e.g.\ in such a way that naturally extends our Christol theorem beyond finite automata. To our naive thought, it will be possible to construct Turing complete dynamical systems in this way (i.e.\ as a consequence of rigidity of some geometric structures), in view of the empirical fact that Turing complete systems are everywhere\footnote{It may be meaningful to recall a solution of Hilbert's 10th problem based on proving that \emph{a set $A \subseteq \nat^n$ of tupples of natural numbers is Diophantine if and only if it is recursively enumerable} (i.e.\ can be enumerated by some Turing machine). See \cite{Hilbert10}.}. In this relation, however, we need to note that the actual goal of such geometric representations of classical computational models is not to give such representations themselves; but to make computational hierarchies more transparent through their geometric representations. Otherwise, we will again end up with a diversity of distinct ad hoc representations, which even contradicts to our goal.
\bibliographystyle{elsarticle-num}
\bibliography{achristol}

\end{document}